\documentclass[aap,MSNbibl,seceqn,dvips]{arximspdf}

\doi{10.1214/12-AAP878} 
\volume{23}
\issue{4}
\pubyear{2013}
\firstpage{1494}
\lastpage{1505}

\makeatletter
\newtheorem{theorem}{Theorem}[section]
\newproclaim{assumption}[theorem]{Assumption}

\newproclaim{definition}[theorem]{Definition}
\newtheorem{lemma}[theorem]{Lemma}
\newtheorem{proposition}[theorem]{Proposition}
\newproclaim{remark}[theorem]{Remark}
\newproclaim{example}[theorem]{Example}

\newcommand{\PP}{\mathbb{P}}
\newcommand{\sint}{\bullet}

\newcommand{\cadlag}{c\`{a}dl\`{a}g}

\newcommand{\E}{\mathbb{E}}
\renewcommand{\P}{\mathbb{P}}
\newcommand{\R}{\mathbb{R}}

\newcommand{\eqref}[1]{(\ref{#1})}
\makeatother

\begin{document}
\begin{frontmatter}

\title{A trajectorial interpretation of Doob's martingale inequalities\thanksref{T1}}
\thankstext{T1}{Supported by the European Research Council (ERC) under
Grant FA506041, by the Vienna Science and Technology Fund (WWTF) under Grant MA09-003,
by the Austrian Science Fund FWF through p21209.}
\runtitle{Trajectorial interpretation: Doob's inequalities}

\begin{aug}
\author[A]{\fnms{B.}~\snm{Acciaio}\corref{}\ead[label=e1]{beatrice.acciaio@univie.ac.at}},
\author[B]{\fnms{M.}~\snm{Beiglb\"ock}},
\author[B]{\fnms{F.}~\snm{Penkner}},
\author[B]{\fnms{W.}~\snm{Schachermayer}}
\and
\author[B]{\fnms{J.}~\snm{Temme}}
\runauthor{B. Acciaio et al.}
\affiliation{University of Vienna and University of Perugia,
University of Vienna,
University of Vienna,
University of Vienna and
University of Vienna}
\address[A]{B. Acciaio\\
Faculty of Mathematics\\
University of Vienna\\
Vienna\\
Austria\\
and\\
Department of Economics\\
\quad Finance and Statistics\\
University of Perugia\\
Perugia\\
Italy\\
\printead{e1}} 
\address[B]{M. Beiglb\"ock\\
F. Penkner\\
W. Schachermayer\\
J. Temme\\
Faculty of Mathematics\\
University of Vienna\\
Vienna\\
Austria}
\end{aug}

\received{\smonth{3} \syear{2012}}
\revised{\smonth{5} \syear{2012}}

%
\begin{abstract}
We present a unified approach to Doob's $L^p$ maximal inequalities for
\mbox{$1\leq p<\infty$}. The novelty of our method is that these
martingale inequalities are obtained as consequences of elementary
\textit{deterministic} counterparts. The latter have
a natural interpretation in terms of robust hedging. Moreover, our
deterministic inequalities lead to new versions of Doob's maximal
inequalities. These are best possible in the sense that equality is
attained by properly chosen martingales.
\end{abstract}

%

\end{frontmatter}

\section{Introduction}
In this paper we derive estimates for the running maximum of a
martingale or nonnegative submartingale in terms of its terminal value.
Given a function $f$ we write $\bar f(t)=\sup_{ u\leq t} f(u)$. Among
other results, we establish the following martingale inequalities.

\begin{theorem}\label{thmdoob}
Let $(S_n)_{n=0}^T$ be a nonnegative submartingale.
Then
%
\renewcommand{\theequation}{Doob-$L^p$}
\begin{equation}
\label{intro1}
\hspace*{-44pt}\E \bigl[\bar S_T^p \bigr]\leq
\biggl(\frac{p}{p-1}\biggr)^p \E\bigl[S_T^p
\bigr] ,\qquad 1<p<\infty,
\end{equation}\vspace*{-6pt}
\renewcommand{\theequation}{Doob-$L^1$}
\begin{equation}
\label{intro2}
\E[\bar S_T]\leq\frac{e}{e-1}\bigl[ \E\bigl[S_T\log(S_T)\bigr]+\E \bigl[S_0
\bigl(1-\log(S_0) \bigr) \bigr]\bigr].
\end{equation}
\end{theorem}

Here \eqref{intro1} is the classical Doob $L^p$-inequality, $p\in(1,
\infty)$~\cite{Do90}, Theorem~3.4. The second result \eqref{intro2}
represents the Doob $L^1$-inequality in the sharp form derived by Gilat
\cite{Gi86} from the $L\log L$~Hardy--Littlewood inequality.

\subsection*{Trajectorial inequalities}

The novelty of this note is that the above martingale inequalities are
established as consequences of \textit{deterministic} counterparts. We
postpone the general statements (Proposition~\ref{propdoobpath}) and
illustrate the spirit of our approach by a simple result that may be
seen as the \textit{trajectorial version} of Doob's $L^2$-inequality.

\textit{Let $s_0, \ldots, s_T$ be real numbers. Then}
%
\renewcommand{\theequation}{Path-$L^2$}
\begin{equation}
\label{DetL2}
\bar{s}_T^2+ 4 \Biggl[ \sum
_{n=0}^{T-1} \bar{s}_n(
s_{n+1}-s_n)\Biggr] \leq4 s_T^2.
\end{equation}
Inequality \eqref{DetL2} is completely elementary and the proof is
straightforward: it suffices to rearrange terms and to complete squares.
The significance of \eqref{DetL2} lies rather in the fact that it
implies \eqref{DoobL2}. Indeed, if $S=(S_n)_{n=1}^T$ is a
nonnegative submartingale, we may apply \eqref{DetL2} to each
trajectory of $S$.
The decisive observation is that, by the submartingale property,
%
\renewcommand{\theequation}{\arabic{section}.\arabic{equation}}
\setcounter{equation}{0}
\begin{equation}
\label{eqsubmart} \E\Biggl[ \sum_{n=0}^{T-1}
\bar S_n( S_{n+1}-S_n)\Biggr]\geq0,
\end{equation}
hence, \eqref{DoobL2} follows from \eqref{DetL2} by taking expectations.

\subsection*{Inequalities in continuous time---sharpness}

Passing to the continuous time setting, it is clear that \eqref
{intro1} and \eqref{intro2} carry over verbatim to the case where
$S=(S_t)_{t\in[0,T]}$ is a nonnegative c\`adl\`ag submartingale, by
the usual limiting argument. It is also not surprising that in
continuous time one has trajectorial counterparts of those
inequalities, the sum in \eqref{DetL2} being replaced by a (carefully
defined) integral.
Moreover, in the case $p=1$, the inequality can be \textit{attained} by a
martingale in continuous time; cf.~\cite{Gi86} and~\cite{GrPe98}.
Notably, this does not hold for $1<p<\infty$. We discuss this for the
case $p=2$ in the $L^2$-norm formulation. Given a nonnegative \cadlag\
submartingale $S=(S_t)_{t\in[0,T]}$, we have
%
\renewcommand{\theequation}{Doob-$L^2$}
\begin{equation}
\label{DoobL2}
\|\bar S_T\|_2 \leq2
\|S_T\|_2.
\end{equation}
Dubins and Gilat~\cite{DuGi78} showed that the constant $2$ in \eqref
{DoobL2} is optimal, that is, cannot be replaced by a strictly
smaller constant. It is also natural to ask whether equality can be
attained in \eqref{DoobL2}. It turns out that this happens only in the
trivial case $S\equiv0$; otherwise, the inequality is strict.
Keeping in mind that equality in \eqref{intro2} is attained, one may
also try to improve on \eqref{DoobL2} by incorporating the starting
value of the martingale. Indeed, we obtain the following result.
%
\begin{theorem}\label{SharkDoob}
$\!\!$For every nonnegative \cadlag\ submartingale $S=(S_t)_{t\in[0,T]}$,
%
\renewcommand{\theequation}{\arabic{section}.\arabic{equation}}
\setcounter{equation}{1}
\begin{equation}
\label{Optimal1} \|\bar S_T\|_2 \leq\|S_T
\|_2+ \|S_T-S_0\|_2.\vadjust{\goodbreak}
\end{equation}
Inequality \eqref{Optimal1} is sharp. More precisely, given
$x_0,x_1\in
\R$, $0< x_0\leq x_1$, there exists a positive, continuous martingale
$S=(S_t)_{t\in[0,T]}$ such that $\|S_0\|_2=x_0,\|S_T\|_2=x_1 $ and
equality holds in \eqref{Optimal1}.
\end{theorem}

In Theorem~\ref{SharkDoobLp} we formulate the result of Theorem~\ref
{SharkDoob} for $1<p<\infty$, thus establishing an optimal a priori
estimate on $\|\bar S_T\|_p$.

We emphasize that the idea that \eqref{intro1} can be improved by
incorporating the starting value $S_0$ into the inequality is not new.
Cox~\cite{Co84}, Burkholder~\cite{Bu91} and Peskir~\cite{Pe98a} show that
%
\renewcommand{\theequation}{\arabic{section}.\arabic{equation}}
\setcounter{equation}{2}
\begin{equation}
\label{CoBuPeIneq} \E\bigl[\bar S_T^2 \bigr]\leq4 \E
\bigl[S_T^2 \bigr] - 2\E\bigl[S_0^2
\bigr].
\end{equation}
Here the constants $4$ (resp., $2$) are sharp (cf.~\cite{Pe98a}) with
equality in \eqref{CoBuPeIneq} holding iff $S\equiv0$.\footnote{That
\eqref{Optimal1} implies \eqref{CoBuPeIneq} follows from a simple
calculation. Alternatively, the sharpness of \eqref{CoBuPeIneq} is a
consequence of the fact that equality in \eqref{Optimal1} can be
attained for all possible values of $\|S_0\|_2, \|S_T\|_2$.}

\subsection*{Financial interpretation}
We want to stress that \eqref{DetL2} has a natural interpretation in
terms of mathematical finance.

Financial intuition suggests that we consider the positive martingale
$S=(S_n)_{n=0}^{T}$ as the process describing the price evolution of
an asset under the so-called ``risk-neutral measure,'' so that $\Phi
(S_0, \ldots, S_T)=(\bar{S}_T)^2$ [resp., $\varphi(S_T)=S_T^2$] have the natural interpretation of a so-called
\textit{exotic} option (resp., a \textit{European} option) written on $S$.
In finance, a European option $\varphi$ (resp., exotic option $\Phi$)
is a function that depends on the final value $S_T$ of $S$ (resp., on
its whole path $S_0,\ldots,S_T$). The seller of the option $\Phi$ pays
the buyer the random amount $\Phi(S_0,\ldots,S_T)$ after its expiration
at time $T$. Following~\cite{Ba00} we may interpret $\mathbb{E}[\Phi]$
as the price that the buyer pays for this option at time $0$; cf. \cite
{Sh04}, Chapter 5, for an introductory survey on risk-neutral pricing.

Here we take the point of view of an economic agent who sells the
option $\Phi$ and wants to protect herself in all possible scenarios
$\omega\in\Omega$, that is, against all possible values $\Phi
(S_0(\omega
),\ldots,S_T(\omega))$, which she has to pay to the buyer of~$\Phi$.
This means that she will trade in the market in order to arrive at time
$T$ with a portfolio value which is at least as high as the value
of $\Phi$. By buying a European option $\varphi(S_T)=S_T^2$, she can
clearly protect herself in case the asset reaches its maximal value at
maturity $T$. However, she still faces the risk of $S$ having its
highest value at some time $n$ before $T$. To protect against that
possibility, one way for her is to ``go short'' in the underlying asset
(i.e., to hold negative positions in $S$). By scaling, her protecting
strategy should be proportional to the running maximum $\bar{S}_n$. At
this point our educated guess is to follow the strategy
$H_n=-4\bar{S}_n$, meaning that from time $n$ to time $n+1$ we keep an
amount $H_n$ of\vadjust{\goodbreak} units of the asset $S$ in our portfolio.
The portfolio strategy produces the following value at time $T$:
%
\begin{equation}
\label{introh} \sum_{n=0}^{T-1}
H_n(S_{n+1}-S_n)=-4\sum
_{n=0}^{T-1} \bar{S}_n(S_{n+1}-S_n).
\end{equation}
The reason why we have chosen the special form $H_n=-4\bar{S}_n$ now
becomes apparent when considering \eqref{DetL2} and \eqref
{eqsubmart}. In
our ``financial mind experiment'' this may be interpreted as follows:
by buying $4$ European options $S_T^2$ and following the
self-financing trading strategy $H=(H_n)_{n=0}^{T-1}$, the seller of
the option $\Phi=(\bar{S}_T)^2$ covers her position at maturity $T$,
whatever the outcome $(S_0(\omega),\ldots,S_T(\omega))$ of the price
evolution is. Thus an upper bound for the price of the exotic option
$\Phi$ in terms of the European option $\varphi$ is given by
\[
\E\bigl[ (\bar{S}_T)^2\bigr]\leq4 \E
\bigl[S_T^2\bigr].
\]

We note that Henry-Labord\`{e}re~\cite{He11b} derived \eqref{intro1}
in a related fashion.

The idea of robust pricing and pathwise hedging of exotic options
seemingly goes back to Hobson~\cite{Ho98b} (see also~\cite{BrHoRo01a,CoOb09,HoKl11}). We refer the reader to~\cite{Ho11} for a thorough
introduction to the topic.

\subsection*{Organization of the paper} In Section~\ref{secproofdoob} we prove Doob's
inequalities \eqref{intro1} and \eqref{intro2} after establishing the
trajectorial counterparts \eqref{eqdoobpath} and \eqref{lempweq}.
We prove Theorem~\ref{SharkDoobLp} and its $L^p$ version in Section~\ref{sec3}.

\section{\texorpdfstring{Proof of Theorem~\protect\ref{thmdoob}}{Proof of Theorem 1.1}}\label{secproofdoob}
The aim of this section is to prove Doob's maximal inequalities in
Theorem~\ref{thmdoob} by means of deterministic inequalities, which
are established in Proposition~\ref{propdoobpath} below. The proof of
Theorem~\ref{thmdoob} is given at the end of this section. Regarding
\eqref{intro1}, we prove the stronger result
%
\begin{equation}
\label{eqstrongdoob} \E \bigl[\bar S_T^p
\bigr]\leq\biggl(\frac{p}{p-1}\biggr)^p \E\bigl[S_T^p
\bigr] -\frac{p}{p-1}\E\bigl[S_0^p\bigr],\qquad 1<p<\infty,
\end{equation}
which was obtained in~\cite{Co84,Pe98a}.

\begin{proposition}\label{propdoobpath}
Let $s_0,\ldots, s_T$ be nonnegative numbers.
\begin{longlist}[(II)]
\item[(I)] For $1< p <\infty$ and $h(x):=-\frac
{p^2}{p-1}x^{p-1}$, we have
%
\renewcommand{\theequation}{Path-$L^p$}
\begin{equation}
\label{eqdoobpath}
\bar s_T^p
\leq\sum_{i=0}^{T-1}h(\bar s_i)
(s_{i+1}-s_i)-\frac
{p}{p-1}s_0^p+
\biggl(\frac{p}{p-1} \biggr)^p s_T^p.
\end{equation}
\item[(II)] For $ h(x):=-\log (x )$, we have
%
\renewcommand{\theequation}{Path-$L^1$}
\begin{equation}
\label{lempweq}
\qquad\qquad\bar s_T\leq\frac{e}{e-1}\Biggl(\sum
_{i=0}^{T-1}h(\bar s_i)
(s_{i+1}-s_i)+s_T\log (s_T
)+s_0\bigl(1-\log(s_0)\bigr)\Biggr).
\end{equation}
\end{longlist}
\end{proposition}
We note that for $p=2$, inequality \eqref{eqdoobpath} is valid also
in the case where $s_0, \ldots, s_T$ are real (possibly negative)
numbers. A continuous time counterpart of \eqref{eqdoobpath} is given
in Remark~\ref{ContPathDoob} below.

In the proof of Proposition~\ref{propdoobpath}, we need the following
identity.
%
\begin{lemma}\label{lempartint}
Let $s_0,\ldots, s_T$ be real numbers and $h\dvtx \R\to\R$ any function. Then
%
\renewcommand{\theequation}{\arabic{section}.\arabic{equation}}
\setcounter{equation}{1}
\begin{equation}
\label{eqsortofpartint} \sum
_{i=0}^{T-1}h(\bar s_i)
(s_{i+1}-s_i)=\sum_{i=0}^{T-1}h(
\bar s_i) (\bar s_{i+1}-\bar s_i) +h(\bar
s_T) (s_T-\bar s_T).
\end{equation}
\end{lemma}
\begin{pf}
This follows by properly rearranging the summands. Indeed, observe that
for a term on the right-hand side there are two possibilities: if $\bar
s_{i+1}=\bar s_i$ (resp., $s_{T}=\bar s_T$) it simply vanishes.
Otherwise, it equals a sum $h(\bar s_{k})(s_{k+1}-s_{k})+\cdots+
h(\bar s_m)(s_{m+1}-s_m)$ where $\bar s_{k}=\cdots=\bar s_m$. In total,
every summand on the left-hand side of \eqref{eqsortofpartint} is
accounted exactly once on the right.\vspace*{-1.5pt}
\end{pf}

We note that Lemma~\ref{lempartint} is a special case of \cite
{ObYo06}, Lemma 3.1.\vspace*{-1pt}
\begin{pf*}{Proof of Proposition~\ref{propdoobpath}}(I)
By convexity, $x^p+px^{p-1}(y-x)\leq y^p, x,y\geq0$. Hence, Lemma~\ref
{lempartint} yields
%
\begin{eqnarray}
\label{fordoobeq} \sum_{i=0}^{T-1}h(
\bar s_i) (s_{i+1}-s_i)&=&-\frac{p^2}{p-1}\sum
_{i=0}^{T-1}\bar s_i^{p-1}(
\bar s_{i+1}-\bar s_i) -\frac{p^2}{p-1}\bar
s_T^{p-1}(s_T-\bar s_T)
\nonumber
\\[-1.5pt]
&\geq&-\frac{p}{p-1}\sum_{i=0}^{T-1}\bar
s_{i+1}^{p}-\bar s_{i}^{p}-
\frac{p^2}{p-1}\bar s_T^{p-1}(s_T-\bar
s_T)
\\[-1.5pt]
&=&p\bar s_T^p-\frac{p^2}{p-1}\bar s_T^{p-1}s_T+
\frac{p}{p-1}\bar s_0^p.
\nonumber
\end{eqnarray}
We therefore have
%
\begin{eqnarray}
\label{eqalgebraiceq} &&\sum_{i=0}^{T-1}h(
\bar s_i) (s_{i+1}-s_i)+ \biggl(
\frac
{p}{p-1} \biggr)^p s_T^p-
\frac{p}{p-1}\bar s_0^p-\bar s_T^p
\nonumber
\\[-6pt]
\\[-6pt]
\nonumber
&&\qquad\geq (p-1)\bar s_T^p-\frac{p^2}{p-1}\bar
s_T^{p-1}s_T+ \biggl(\frac
{p}{p-1}
\biggr)^p s_T^p.
\end{eqnarray}
To establish \eqref{eqdoobpath} it is thus sufficient to show that
the right-hand side of \eqref{eqalgebraiceq} is nonnegative.
Defining $c$ such that $S_n=c\bar S_n$ amounts to showing that\vspace*{-1pt}
%
\begin{equation}
\label{gc} g(c)=(p-1)-\frac{p^2}{p-1}c+ \biggl(\frac{p}{p-1}
\biggr)^p c^p\geq0.\vspace*{-1pt}
\end{equation}
Using standard calculus we obtain that $g$ reaches its minimum at $\hat
c= \frac{p-1}p$ where $g(\hat c)=0$.

(II) By Lemma~\ref{lempartint} we have
\begin{eqnarray*}
&&\sum_{i=0}^{T-1}h(\bar s_i)
(s_{i+1}-s_i)\\[-1.5pt]
&&\qquad= -\sum_{i=0}^{T-1}
\log (\bar s_i) (\bar s_{i+1}-\bar s_i)-\log(
\bar s_T) (s_{T}-\bar s_T)
\\[-1.5pt]
&&\qquad\geq\sum_{i=0}^{T-1} \bigl(\bar
s_{i+1}-\bar s_i-\bar s_{i+1}\log (\bar
s_{i+1} )+\bar s_{i}\log (\bar s_{i} ) \bigr)-\log
(\bar s_T ) (s_{T}-\bar s_T)
\\[-1.5pt]
&&\qquad=\bar s_T-s_0+s_0\log(s_0)-s_T
\log(\bar s_T),
\end{eqnarray*}
where the inequality follows from the convexity of $x\mapsto-x+x\log
(x )$, \mbox{$x>0$}. If $s_T=0$ then the above inequality shows that
\eqref{lempweq} holds true. Otherwise, we have
\[
\bar s_T\leq\sum_{i=0}^{T-1}h(
\bar s_i) (s_{i+1}-s_i)+s_0-s_0
\log (s_0)+s_T\log(s_T)+s_T\log
\biggl(\frac{\bar s_T}{s_T} \biggr).
\]
Note that the function $x\mapsto x\log (y/x )$ on
$(0,\infty
)$, for any fixed $y>0$, has a maximum in $\hat x=y/e$, where it takes
the value $y/e$. This means that $s_T\log ({\bar s_T}/{s_T}
)\leq{\bar s_T}/{e}$ which concludes the proof.
\end{pf*}

We are now in the position to prove Theorem~\ref{thmdoob}.
\begin{pf*}{Proof of Theorem~\ref{thmdoob}} By Proposition~\ref
{propdoobpath}(I), for $h(x):=-\frac{p^2}{p-1}x^{p-1}$
we have
%
\begin{equation}
\label{InProofDoobSorryForStupidLabels} \bar S_T^p\leq
\sum_{i=0}^{T-1}h(\bar S_i)
(S_{i+1}-S_i)-\frac
{p}{p-1}S_0^p+
\biggl(\frac{p}{p-1} \biggr)^p S_T^p.
\end{equation}
Since $S$ is a submartingale and $h$ is negative, $\E[\sum_{i=0}^{T-1}h(\bar S_i)(S_{i+1}-S_i)]\leq0$ and thus \eqref
{eqstrongdoob} [and consequently \eqref{intro1}] follows from
\eqref
{InProofDoobSorryForStupidLabels} by taking expectations.

Inequality \eqref{intro2} follows from Proposition~\ref
{propdoobpath}(II) in the same fashion.
\end{pf*}

\begin{remark}
Given the terminal law $\mu$ of a martingale $S$, Hobson~\cite{Ho11},
Section $3.7$, also provides pathwise hedging strategies for lookback
options on $S$. As opposed to the strategies given in Proposition $\ref
{propdoobpath}$, we emphasize that the strategies in~\cite{Ho11}
depend on $\mu$.
\end{remark}

\section{\texorpdfstring{Qualitative Doob $L^p$-inequality---Proof of Theorem \protect\ref{SharkDoob}}
{Qualitative Doob L p-inequality---Proof of Theorem 1.2}}\label{sec3}

In this section we prove Theorem~\ref{SharkDoob} as well as the
following result which pertains to $p\in(1,\infty)$.

\begin{theorem}\label{SharkDoobLp}
Let $(S_t)_{t\in[0,T]}$ be a nonnegative submartingale, $S\neq0$
and $1<p<\infty$. Then
%
\begin{equation}
\label{eqSharkDoobLp} \|\bar S_T\|_p\leq
\frac{p}{p-1}\|S_T\|_p-\frac{1}{p-1}
\frac{\|S_0\|^p_p}{\|\bar S_T\|_p^{p-1}}.
\end{equation}

Given the values $\|S_0\|_p$ and $\|S_T\|_p$, inequality \eqref
{eqSharkDoobLp} is best possible. More precisely, given $x_0,x_1\in\R
$, $0< x_0\leq x_1$, there exists a positive, continuous martingale
$S=(S_t)_{t\in[0,T]}$ such that $\|S_0\|_p=x_0,\|S_T\|_p=x_1 $ and
equality holds in \eqref{eqSharkDoobLp}.

Moreover, equality in \eqref{eqSharkDoobLp} holds if and only if $S$
is a nonnegative martingale such that $\bar S$ is continuous and $\bar
S_T= \alpha S_T$, where $\alpha\in[1,\frac{p}{p-1})$.
\end{theorem}
%
\begin{remark}\label{remidentity}
We prove Theorem~$\ref{SharkDoobLp}$ by introducing a pathwise
integral in continuous time. Note that inequality \eqref
{eqSharkDoobLp} can also be obtained without defining such an
integral. However, the definition of the pathwise integral will allow
us to characterize all submartingales for which equality in \eqref
{eqSharkDoobLp} holds.
\end{remark}

\subsection*{Connection between Theorems \protect\ref{SharkDoob} and~\protect\ref{SharkDoobLp}} We now discuss under which
conditions Theorem~\ref
{SharkDoob} and Theorem~\ref{SharkDoobLp} are equivalent for $p=2$.
Recall that Theorem~\ref{SharkDoob} asserts that
%
\begin{equation}
\label{Optimalfinal1} \|\bar S_T\|_2 \leq
\|S_T\|_2+ \|S_T- S_0
\|_2
\end{equation}
and Theorem~\ref{SharkDoobLp} reads, in the case of $p=2$, as
%
\begin{equation}
\label{eqSharkDoobLp1} \|\bar S_T\|_2 \leq2
\|S_T\|_2- \frac{\|S_0\|_2^2}{\|\bar S_T\|_2}.
\end{equation}
\begin{itemize}
\item If $S$ is a \textit{martingale}, then \eqref{Optimalfinal1} and
\eqref{eqSharkDoobLp1} are equivalent. Indeed, rearranging \eqref
{eqSharkDoobLp1} yields
%
\begin{equation}
\label{eqpsi} \psi\bigl(\|\bar S_T\|_2\bigr):=\frac{1}{2}\|
\bar S_T\|_2 + \frac{\|S_0\|_2^2}{2\|\bar S_T\|_2} \leq\|S_T
\|_2,
\end{equation}
and by inverting the strictly monotone function $\psi$ on $[\|S_0\|_2,\infty)$, we obtain
\[
\|\bar S_T\|_2 \leq\psi^{-1}\bigl(\|S_T
\|_2\bigr)=\|S_T\|_2+ \sqrt{
\|S_T\|_2^2- \|S_0
\|_2^2}.
\]
Since $S$ is a martingale, $\sqrt{ \|S_T\|_2^2- \|S_0\|_2^2}=\|
S_T-S_0\|_2$, which gives \eqref{Optimalfinal1}.\vadjust{\goodbreak}
\item If $S$ is a true submartingale, then the estimate in \eqref
{Optimalfinal1} is in fact stronger than~\eqref{eqSharkDoobLp1}. This
follows from the above reasoning and the fact that for a true
submartingale, we have $\sqrt{ \|S_T\|_2^2- \|S_0\|_2^2}> \|S_T-S_0\|_2$.
\item Clearly, it would be desirable to also obtain for general $p$ an
inequality of the type \eqref{Optimalfinal1}, which is in the case of a
martingale $S$ equivalent to \eqref{eqSharkDoobLp}, and where $\bar
S_T$ only appears on the left-hand side. By similar reasoning as for
$p=2$, finding such an inequality is tantamount to inverting the function
\[
\psi(x)=\frac{p-1}{p}x+\frac{\|S_0\|_p^p}{px^{p-1}},
\]
which is strictly monotone on $[\|S_0\|_p,\infty)$. Since finding
$\psi^{-1}$ amounts to solving an algebraic equation, there is, in general,
no closed form representation of $\psi^{-1}$ unless $p\in\{2,3,4\}$.
\end{itemize}

\subsection*{Definition of the continuous-time integral}

For a general account on the theory of pathwise stochastic integration
we refer to Bichteler~\cite{Bi81} and Karandi\-kar~\cite{Ka95}. Here we
are interested in the particular case where the integrand is of the
form $h(\bar S )$ and $h$ is monotone and continuous. In this setup a
rather naive and ad hoc approach is sufficient (see Lemma~\ref
{lemhonfoell} below).

Fix \cadlag\ functions $f,g\dvtx [0,T]\to[0,\infty)$ and assume that $g$ is
monotone. We set
%
\begin{equation}
\label{eqdefintcontgen} \int
_0^T g_{t\mbox{-}} \,df_t:=
\lim_{n\to\infty}\sum_{t_i\in\pi
_n}g_{t_i\mbox{-}}
(f_{t_{i+1}}-f_{t_i} )
\end{equation}
if the limit exists for \textit{every} sequence of finite partitions
$\pi_n$ with mesh converging to $0$. The standard argument of mixing
sequences then implies uniqueness. We stress that \eqref
{eqdefintcontgen} exists if and only if the ``nonpredictable
version'' $\int_0^T g_{t} \,df_t=\lim_{n\to\infty}\sum_{t_i\in\pi
_n}g_{t_i} (f_{t_{i+1}}-f_{t_i} )$ exists; in this case the
two values coincide.

By rearranging terms, one obtains the identity
%
\begin{eqnarray}
\label{eqpartintdisc} \sum
_{t_i\in\pi} g_{t_i}(f_{t_{i+1}}-f_{t_i})&=&-
\sum_{t_i\in\pi
}f_{t_i}(g_{t_{i+1}}-g_{t_{i}})+g_Tf_T-g_0f_0
\nonumber
\\[-8pt]
\\[-8pt]
\nonumber
&&{}-
\overbrace{\sum_{t_i\in
\pi
}(g_{t_{i+1}}-g_{t_i})
(f_{t_{i+1}}-f_{t_i})}^{(*)}.
\end{eqnarray}
If it is possible to pass to a limit on either of the two sides, one
can do so on the other. Hence, $\int_0^T g_t \,df_t$ is defined whenever
$\int_0^T f_t \,dg_t$ is defined and vice versa, since the monotonicity\vadjust{\goodbreak}
of $g$ implies that $(*)$ converges. In this case we obtain the
integration-by-parts formula
%
\begin{equation}
\label{eqpartintcont} \qquad\int_0^T
g_t \,df_t=-\int_0^T
f_t \,dg_t+g_Tf_T-g_0f_0-
\sum_{0\leq t\leq
T}(g_t-g_{t\mbox{-}})
(f_t-f_{t\mbox{-}}).
\end{equation}

Below we will need that the integrals $\int_0^T h(\bar f_t) \,df_t$ and
$ \int_0^T f_t \,dh(\bar f_t)$ are well-defined whenever $h$ is
continuous, monotone and $f$ is \cadlag. In the case of $ \int_0^T f_t
\,dh(\bar f_t)$, this can be seen by splitting $f$ in its continuous
and its jump part. Existence of $\int_0^T h(\bar f_t) \,df_t$ is then a
consequence of \eqref{eqpartintcont}.

The following lemma establishes the connection of the just defined
pathwise integral with the standard It\^o integral.

\begin{lemma}\label{lemhonfoell}
Let $S$ be a martingale on $(\Omega,\mathcal F,(\mathcal F_t)_{t\geq
0},\PP)$ and $h$ be a monotone and continuous function. 
Then
%
\begin{equation}
\label{eqdefintcontstoch} \bigl(h(\bar
S)\sint S\bigr)_T(\omega)=\int_0^Th
\bigl(\bar S_{t\mbox{-}}(\omega)\bigr) \,dS_t(\omega) \qquad\P
\mbox{-a.s.},
\end{equation}
where the left-hand side refers to the It\^o integral while the
right-hand side appeals to the pathwise integral defined in \eqref
{eqdefintcontgen}.
\end{lemma}
\begin{pf}
Karandikar~\cite{Ka95}, Theorem 2, proves that
\[
\bigl(h(\bar S)\sint S\bigr)_T(\omega)=\lim_{n\to\infty}\sum
_{t_i\in\pi
_n}h\bigl(\bar S_{t_i\mbox{-}}(\omega)\bigr)
\bigl(S_{t_{i+1}}(\omega)-S_{t_i}(\omega ) \bigr)
\]
for a suitably chosen sequence of random partitions $\pi_n, n\geq1$.
According to the above discussion,
$\int_0^T\!h(\bar S_{t\mbox{-}}(\omega)) \,dS_t(\omega)\!=\!\lim_{n\to
\infty
}\sum_{t_i\in\pi_n}\!h(\bar S_{t_i\mbox{-}}(\omega))
(S_{t_{i+1}}(\omega)-S_{t_i}(\omega) )$ for any choice of
partitions $\pi_n(\omega), n\geq1$, with mesh converging to $0$.
\end{pf}

We are now able to establish a continuous-time version of
Proposition~\ref{propdoobpath}.

\begin{proposition}\label{propdoobpathcont}
Let $f\dvtx [0,T]\to[0,\infty)$ be \cadlag. Then for $h(x):=-\frac
{p^2}{p-1}x^{p-1}$, we have
%
\begin{equation}
\label{eqdoobLppath} \bar f_T^p
\leq\int_0^Tp^{-1}h(\bar
f_t) \,df_t+\frac{p}{p-1}\bar f_T^{p-1}f_T-
\frac{1}{p-1} f_0^p.
\end{equation}
Equality in \eqref{eqdoobLppath} holds true if and only if $\bar f$
is continuous.
Similarly, a continuous-time version of \eqref{lempweq} also holds true.
\end{proposition}

\begin{pf}
Inequality \eqref{eqdoobLppath} follows from \eqref{fordoobeq} by
passing to limits. We now show that equality in \eqref{eqdoobLppath}
holds iff $\bar f$ is continuous. To simplify notation, we consider the
case $p=2$. Formula \eqref{eqpartintcont} implies
%
\begin{eqnarray}\label{eqdoobLppath1}
\int_0^Th(\bar f_t)
\,df_t&=&4\int_0^T f_t
\,d\bar f_t-4\bar f_Tf_T+4f_0^2\nonumber\\
&&{}+4
\sum_{0\leq t\leq T}(\bar f_t-\bar
f_{t\mbox
{-}}) (f_t-f_{t\mbox{-}}),
\\
&\geq&2\bar f_T^2-4\bar f_Tf_T+2
\bar f_0^2,\nonumber
\end{eqnarray}
where equality in \eqref{eqdoobLppath1} holds iff $\bar f$ is
continuous. Hence, equality in \eqref{eqdoobLppath} holds true iff
$\bar f$ is continuous.
\end{pf}
If we choose $f$ to be the path of a continuous martingale, the
integral in \eqref{eqdoobLppath} is a pathwise version of an Az\'
{e}ma--Yor process; cf.~\cite{ObYo06}, Theorem 3.
%
\begin{remark}\label{ContPathDoob}
Passing to limits in \eqref{eqdoobpath} in Section $\ref
{secproofdoob}$ we obtain that for every \cadlag\ function
$f\dvtx [0,T]\to
[0,\infty)$
\[
\bar f_T^p \leq-\int_0^T
\frac{p^2}{p-1}\bar f_t^{p-1} \,df_t+ \biggl(
\frac{p}{p-1} \biggr)^p f_T^p -
\frac{p}{p-1}f_0^p,\qquad 1<p<\infty.
\]
Alternatively, this can be seen as a consequence of \eqref{eqdoobLppath}.
\end{remark}

\begin{lemma}\label{QualLp}
Let $(S_t)_{t\in[0,T]}$ be a nonnegative submartingale and
$1<p<\infty$. Set $S=M+A$, where $M$ is a martingale and $A$ is an
increasing, predictable process with $A_0=0$. Then
%
\begin{equation}
\label{eqQualLp} \quad\E\bigl[\bar S_T^p\bigr]\leq-
\frac{p}{p-1}\E\bigl[S_0^{p-1}A_T\bigr]+
\frac{p}{p-1}\E \bigl[\bar S_T^{p-1}S_T\bigr]
-\frac{1}{p-1} \E\bigl[S_0^p\bigr].
\end{equation}
Equality holds in \eqref{eqQualLp} if and only if $S$ is a martingale
such that $\bar S$ is a.s. continuous.
\end{lemma}
\begin{pf}
By Proposition~\ref{propdoobpathcont} we find for $h(x)=-\frac
{p^2}{p-1}x^{p-1}$,
%
\begin{equation}
\label{eqQualLpproof1} \bar S_T^p \leq\int
_0^Tp^{-1}h(\bar S_t)
\,dS_t+\frac{p}{p-1}\bar S_T^{p-1}S_T-
\frac{1}{p-1} S_0^p,
\end{equation}
where equality holds iff $\bar S$ is continuous. Since
%
\begin{equation}
\label{eqQualLpproof2} \E \biggl[\int_0^Tp^{-1}h(
\bar S_t) \,dA_t \biggr]\leq-\frac{p}{p-1}\E
\bigl[S_0^{p-1}A_T\bigr],
\end{equation}
\eqref{eqQualLp} follows by taking expectations in \eqref
{eqQualLpproof1}. As the estimate in \eqref{eqQualLpproof2} is an
equality iff $A=0$, we conclude that equality in \eqref{eqQualLp}
holds iff $S$ is a martingale such that $\bar S$ is continuous.
\end{pf}

We note that in the case of $p=2$,~\cite{AzYo79}, Corollary 2.2.2, also
implies that equality in \eqref{eqQualLp} holds for every continuous
martingale $S$.

\begin{pf*}{Proof of Theorems~\ref{SharkDoobLp} and~\ref
{SharkDoob}}
By Lemma~\ref{QualLp} and H\"{o}lder's inequality we have
%
\begin{eqnarray}
\label{PostCS1}\qquad \|\bar S_T\|_p^p &\leq&-
\frac{p}{p-1}\E\bigl[S_0^{p-1}A_T\bigr]+
\frac{p}{p-1}\bigl\| \bar S_T^{p-1}S_T
\bigr\|_1-\frac{1}{p-1}\|S_0\|_p^p
\\
\label{PostCS} &\leq&-\frac{p}{p-1}\E\bigl[S_0^{p-1}A_T
\bigr]+\frac{p}{p-1}\|\bar S_T\|^{p-1}_p
\|S_T\|_p-\frac{1}{p-1}\|S_0
\|_p^p,
\end{eqnarray}
where equality in \eqref{PostCS1} holds for every martingale $S$ such
that $\bar S$ is continuous, and equality in \eqref{PostCS} holds
whenever $S_T$ is a constant multiple of $\bar S_T$.
Since $\E[S_0^{p-1}A_T]\geq0$, we obtain \eqref{eqSharkDoobLp} after
dividing by $\|\bar S_T \|_p^{p-1}$.

In order to establish \eqref{Optimal1} in Theorem~\ref{SharkDoob} for
$p=2$, we rearrange terms in~\eqref{PostCS} to obtain
\[
\psi\bigl(\|\bar S_T\|_2\bigr):=\frac{1}{2}\|\bar
S_T\|_2 + \frac{2\E
[S_0A_T]+\|
S_0\|_2^2}{2\|\bar S_T\|_2} \leq\|S_T
\|_2.
\]
Similarly, as in the discussion after Remark~\ref{remidentity} above,
inverting $\psi$ on $[\|S_0\|_2,\infty)$ implies
\[
\|\bar S_T\|_2\leq\|S_T\|_2 +
\sqrt{\|S_T\|_2^2-2
\E[S_0A_T]-\|S_0\|_2^2}.
\]
Since for every submartingale $S$ we have $\sqrt{\|S_T\|_2^2-2\E
[S_0A_T]-\|S_0\|_2^2}=\|S_T-S_0\|_2$, this proves \eqref{Optimal1}.

In order to prove that \eqref{eqSharkDoobLp} [resp., \eqref
{Optimal1}] is attained, we have to ensure the existence of a
$p$-integrable martingale $S$ such that $\bar S$ is continuous and
$S_T$ is a constant multiple of $\bar S_T$. To this end, we may clearly
assume that $x_0=1$. Fix $\alpha\in(1,\frac{p}{p-1})$ and let
$B=(B_t)_{t\geq0}$ be a Brownian motion starting at $B_0=1$. Consider
the process
$B^{\tau_\alpha}=(B_{t\wedge\tau_\alpha})_{t\geq0}$ obtained by
stopping $B$ at the stopping time
\[
\tau_\alpha:=\inf \{t>0 \dvtx  B_t\leq\bar B_t/
\alpha \}.
\]
This stopping rule corresponds to the Az\'{e}ma--Yor solution of the
Skorokhod embedding problem $(B,\mu)$ (cf.~\cite{AzYo79}) where the
probability measure $\mu$ is given by
\[
\frac{d\mu}{dx}=\frac{\alpha^{-{1}/{(\alpha-1)}}}{(\alpha
-1)}x^{-{(2\alpha-1)}/{(\alpha-1)}}{1}_{[\alpha^{-1},\infty)}(x).
\]
Clearly $B^{\tau_\alpha}$ is a uniformly integrable martingale.
Therefore the process $(S_t)_{t\in[0,T]}$ defined as $
{S_t:=B_{{t}/{(T-t)}\wedge\tau_\alpha}}$ is a nonnegative martingale
satisfying $S_T=\bar S_T/\alpha$. $S_T$ is $p$-integrable for $\alpha
\in
(1,\frac{p}{p-1})$ and $\|S_T\|_p$ runs through the interval
$(1,\infty
)$ while~$\alpha$ runs in $(1,\frac{p}{p-1})$. This concludes the proof.

In fact, note that the proof shows that equality in \eqref
{eqSharkDoobLp} holds if and only if~$S$ is a nonnegative martingale
such that $\bar S$ is continuous and $\bar S_T= \alpha S_T$, where
$\alpha\in[1,\frac{p}{p-1})$.
\end{pf*}

\section*{Acknowledgment}
The authors thank Jan Ob\l{}oj for insightful comments
and remarks.

%


\printaddresses

\end{document}